\documentclass[12pt,final,a4paper]{article}

\usepackage{anziamjedraft,amssymb}

\usepackage[color,notref,notcite]{showkeys}
\newcounter{cst}
\catcode`\@=11
\def\ctel#1{C_{\refstepcounter{cst}\@bsphack
\protected@write\@auxout{}%
           {\string\newlabel{#1}{{\thecst}{\thepage}}}\thecst}}
\catcode`\@=12
\newcommand{\cter}[1]{C_{\ref{#1}}}

\def\N{\mathbb{N}}
\def\R{\mathbb{R}}
\def\O{\Omega}
\def\div{\mbox{div}}
\def\bu{\overline{u}}
\def\dsp{\displaystyle}
\def\n{\mathbf{n}}

\def\bpsi{{\boldsymbol \psi}}
\def\bchi{{\boldsymbol \chi}}

\title{Introduction to discrete functional analysis techniques for the numerical study of 
diffusion equations with irregular data}

\author{J. Droniou}
\address{School of Mathematical Sciences, Monash University, 
Clayton, VIC 3800, \textsc{Australia}.}
\mailto{jerome.droniou@monash.edu}

\date{\today}

\begin{document}

\maketitle

\begin{abstract}
We give an introduction to discrete functional analysis techniques for stationary and transient
diffusion equations.
We show how these techni\-ques are used to establish the convergence of various numerical schemes
without assuming non-physical regularity on the data. For simplicity of exposure, 
we mostly consider linear elliptic equations, and we briefly explain how these techniques
can be adapted and extended to non-linear time-dependent meaningful models (Navier--Stokes equations,
flows in porous media, etc.). These convergence techniques rely on discrete
Sobolev norms and the translation to the discrete setting of functional analysis results.
\end{abstract}

\tableofcontents

\section{Introduction}
\label{sec:intro}

A number of real-world problems are modelled by partial differential equations (\textsc{pde}s)
which involve some form of singularity.
For example, oil engineers deal with underground reservoirs made of
stacked geological layers with different rock properties, which translate
into discontinuous data (permeability tensor, porosity, etc.) in the corresponding mathematical model. Another example from
reservoir engineering is the modelling of wellbores; the relative scales of
the wellbores ($\sim$10-20cm in diameter) and the reservoir ($\sim$1-2km large)
justifies representing injection and production terms at the wells
by Radon measures \cite{fg00}. The mathematical analysis of \textsc{pde}s
involving singular data is challenging. The meanings of the terms in the equations
have to be re-thought; classical derivatives can no longer be used,
and weak/distribution derivatives and Sobolev spaces must be introduced \cite{brezis}.
Beyond these now well-known tools, other techniques
had to be developed for the most complex models to define appropriate notions of solutions,
and to prove their existence (and uniqueness, if possible): renormalised solutions \cite{DMOP99},
entropy solutions \cite{BBGGPV95}, monotone operators and semi-groups \cite{brezis},
elliptic and parabolic capacity \cite{DMOP99,DPP03}, etc.
The main purpose of this analysis is to ensure that the models are well-posed,
that is that they make sense from a mathematical perspective. It is rarely possible
to give explicit forms for, or even detailed qualitative behaviour of, the
solutions to the extremely complex models involved in field
applications. Precise quantitative information that can be used for
decision-making can be obtained only through numerical approximation.

The role of mathematics in obtaining accurate approximate solutions to \textsc{pde}s is twofold.
First, algorithms have to be designed to compute these solutions. But, even based on
sound reasoning, in some circumstances algorithms can fail to approximate the expected
model \cite[Chap. III, Sec. 3]{phdfaille}. Benchmarking
(testing the algorithms in well-documented cases) is useful to ensure the quality of
numerical methods, but it  cannot cover all situations that may occur in field applications.
The second role of mathematics in the numerical approximation of real-world models
is to provide rigorous
analysis of the properties and convergence of the schemes; this analysis
is not restricted to particular cases, and is essential to ensure the reliability of numerical
methods for \textsc{pde}s. 

The usual way to prove the convergence of a scheme is to establish error
estimates; if $\bu$ is the solution to the \textsc{pde} and $u_h$ is the solution provided by the
scheme (where $h$ is, for example, the mesh size), then one will try to establish a bound of the kind
\begin{equation}\label{error}
||u_h-\bu||_X \le \mathcal C h^\alpha
\end{equation}
where $||\cdot||_X$ is an adequate norm and $\alpha>0$. Such
an inequality provides an estimate on the $h$ that must be selected
in order to achieve a pre-determined accuracy of the approximation. However, major
limitations exist:
\begin{itemize}
\item Estimates of the kind \eqref{error} can be established only if the uniqueness
of the solution $\bu$ to the \textsc{pde} is known (if \eqref{error} holds, then $\bu$ is unique
and, actually, the proof of \eqref{error} often mimics a proof of uniqueness of $\bu$).
\item The constant $\mathcal C$ usually depends on higher derivatives of $\bu$ or the
\textsc{pde} data, and \eqref{error} therefore requires some regularity assumptions on the solution or data.
\end{itemize}
For many non-linear real-world models, including those from reservoir engineering
\cite{F95} and the famous Navier--Stokes equations, uniqueness
of the solution is not known unless strong regularity properties on the solution are assumed.
These properties cannot be established in field applications. Hence
convergence analysis based on error estimates is doomed to be somewhat disconnected from
applications. 
This article presents an introduction to techniques that were recently developed
to deal with this issue. These techniques enable the convergence analysis of
numerical schemes under assumptions that are compatible with real-world data and constraints.

Section \ref{sec:cvcp} details the
convergence technique on a simple linear stationary diffusion equation.
After recalling some basic energy estimates on the model, we present the
general path (in Section \ref{ssec:path}) to establish
the convergence of schemes without any regularity assumptions on the data;
this path relies on compactness techniques and \emph{discrete functional analysis} tools, translations
to the discrete setting of functional analysis results pertaining to functions
of continuous variables. Section \ref{ssec:ex} shows on two particular
schemes (two-point finite volume scheme, and non-conforming $\mathbb{P}^1$ finite element
scheme) how this path is applied in practice.
In Section \ref{sec:nl} we discuss the extension of this convergence technique
to non-linear and non-stationary models, more realistic representations of
physical phenomena. We briefly show that virtually no adaptation is required
from the technique used in the linear setting to deal with the simplest non-linear
models. We then give a brief overview of physical models whose numerical analysis
was successfully tackled using discrete functional analysis tools. These include
the Navier--Stokes equations, \textsc{pde}s involved in glaciology, models of oil recovery, models of
melting materials, etc.

\section{Convergence by compactness techniques}\label{sec:cvcp}

\subsection{Model and preliminary considerations}

Let us consider, for our initial presentation, the linear diffusion equation
\begin{equation}\label{base}
\left\{
\begin{array}{ll}
-\div(A\nabla \bu)=f&\mbox{ in $\O$},\\
\bu=0&\mbox{ on $\partial\O$}.
\end{array}\right.
\end{equation}
In the context of reservoir engineering, \eqref{base} corresponds to a steady single-phase
single-component Darcy problem with no gravitational effects \cite{DIP13-2}; $\bu$ is
the pressure and $A$
is the matrix-valued permeability field. This field is usually considered
piecewise constant (constant in each geological layer), and it is therefore discontinuous.
Equation \eqref{base} cannot be considered under the classical sense -- with 
$\div$ and $\nabla$ denoting standard derivatives -- and must be re-written
in a weak form; this form is obtained by multiplying the equation by a test function $v$ which vanishes
on $\partial\Omega$ and by using Stokes' formula \cite{brezis}:
\begin{equation}\label{basew}
\left\{\begin{array}{l}
\mbox{Find $\bu\in H^1_0(\O)$ such that:}\\
\dsp \forall v\in H^1_0(\O)\,,\;\int_\O A(x)\nabla \bu(x)\cdot\nabla v(x)d x
=\int_\O f(x)v(x)d x.
\end{array}\right.
\end{equation}
Here, $H^1_0(\O)$ is the Sobolev space of functions $v\in L^2(\O)$ (square-integrable functions,
equipped with the norm $||v||_{L^2(\O)}^2=\int_\O |v(x)|^2d x$), 
that have a weak (distribution) gradient $\nabla v$ in $L^2(\O)^d$ and a zero value (trace) on 
$\partial\O$. Under the following assumptions, all terms in \eqref{basew} are well-defined:
\begin{align}
&
\begin{array}{l}
\mbox{$\O$ is a bounded open set of $\R^d$ ($d\ge 1$) and $f\in L^2(\O)$},
\end{array}\label{hypO}\\
&\begin{array}{l}
A:\O\mapsto \mathcal M_d(\R)\mbox{ is a measurable matrix-valued mapping,}\\
\exists 0<\underline{a}\le\overline{a}<\infty\mbox{ such that }
|A(x)\xi|\le \overline{a}|\xi|\mbox{ and }A(x)\xi\cdot\xi\ge \underline{a}|\xi|^2\\
\mbox{for almost every $x\in\O$ and all $\xi\in\R^d$}.
\end{array} \label{hypa}
\end{align}
Here, $|\cdot|$ is the Euclidean norm on $\R^d$.
By taking $v=\bu$ in \eqref{basew} and by applying Cauchy-Schwarz' inequality on the right-hand side,
we find
\begin{multline}\label{est1}
\underline{a}||\,|\nabla \bu|\,||_{L^2(\O)}^2\le \int_\O A(x)\nabla\bu(x)\cdot\nabla\bu(x)d x
=\int_\O f(x)\bu(x)d x\\
\le ||f||_{L^2(\O)}||\bu||_{L^2(\O)}.
\end{multline}
Essential to the analysis of elliptic equations is Poincar\'e's inequality:
\begin{equation}\label{poinc}
\forall v\in H^1_0(\O)\,,\;||v||_{L^2(\O)}\le {\rm diam}(\O)||\,|\nabla v|\,||_{L^2(\O)}.
\end{equation}
Substituted into \eqref{est1}, this inequality leads to the following energy estimate, in which
the left-hand side defines the norm in $H^1_0(\O)$:
\begin{equation}\label{energ}
||\bu||_{H^1_0(\O)}:=||\,|\nabla\bu|\,||_{L^2(\O)}\le {\rm diam}(\O)\underline{a}^{-1}||f||_{L^2(\O)}.
\end{equation}

\subsection{General path for the convergence analysis}\label{ssec:path}

Estimate \eqref{energ} shows that $H^1_0(\O)$ is the natural energy space of Problem \eqref{base}.
This estimate is at the core of the theoretical study of \eqref{base} and its non-linear variants,
partly due to Rellich's compactness theorem \cite{brezis}.


\begin{theorem}[Rellich's compact embedding] If $\O$ is a bounded subset of $\R^d$, $d\ge 1$,
and if $(v_n)_{n\in\N}$ is bounded in $H^1_0(\O)$, then $(v_n)_{n\in\N}$ has a subsequence
that converges in $L^2(\O)$. Furthermore, any limit in $L^2(\O)$ of a subsequence of
$(v_n)_{n\in\N}$ belongs to $H^1_0(\O)$.
\label{th-rellich}\end{theorem}
This theorem justifies the general path for a convergence analysis that
is applicable without smoothness assumption on the data or the solution, and that can
be adapted to non-linear equations. As described by Droniou \cite{review}, this path comprises
three steps:
\begin{enumerate}
\item\label{step1} Establish a priori energy estimates si\-milar to \eqref{energ} on the solutions to the scheme,
in a mesh- and scheme-dependent discrete norm that mimics the $H^1_0$ norm,
\item\label{step2} Prove a compactness result, discrete equivalent of Theorem \ref{th-rellich}:
if $(u_h)_h$ is a sequence of discrete functions that are
bounded in the norms introduced in Step~\ref{step1}, then as the mesh size $h$
goes to zero there is a subsequence of $(u_h)_h$ that converges (at least in $L^2(\O)$)
to a function $\bu\in H^1_0(\Omega)$,
\item\label{step3} Prove that if $\bu\in H^1_0(\O)$ is the limit in $L^2(\O)$ as $h\to 0$
of solutions to the scheme, then $\bu$ satisfies \eqref{basew}.
\end{enumerate}

\begin{remark}
The existence of a solution to the \textsc{pde} does not need to
be known. It is obtained as a consequence of the convergence proof.
\end{remark}

The discrete $H^1_0(\O)$ norm is dictated by the scheme. It must be a norm for which
(i) a priori estimates on the numerical solutions can be obtained, and (ii)
the compactness result in Convergence Step \ref{step2} holds. There is however
a norm applicable to a number of numerical methods.
Let us assume that $\O$ is polytopal (polygonal in 2D, polyhedral in 3D, etc.),
and that $\mathcal M$ is a mesh of $\O$ made of polytopal cells. We denote by
$h_{\mathcal M}=\max_{K\in\mathcal M}{\rm diam}(K)$ the size of $\mathcal M$,
and by $X_{\mathcal M}$ the space of piecewise constant functions in the cells.
We identify $v\in X_{\mathcal M}$ with the family of its values $(v_K)_{K\in\mathcal M}$ in the cells.
$\mathcal E_{\mathcal M}$ is the set of all faces of the mesh (edges in 2D),
and $|\sigma|$ denotes the $(d-1)$-dimensional measure of a face $\sigma$ (i.e. length
in 2D, area in 3D).
We take one point $x_K$ in each cell $K$, and we let $d_{K,\sigma}={\rm dist}(x_K,\sigma)$
(see Figure \ref{fig:mesh}).
If $\sigma$ is an interface between two cells $K$ and $L$, then we define
$d_\sigma=d_{K,\sigma}+d_{L,\sigma}$; otherwise, $d_\sigma=d_{K,\sigma}$
with $K$ the unique cell whose $\sigma$ is an face. 

\begin{figure}[h]
\begin{center}
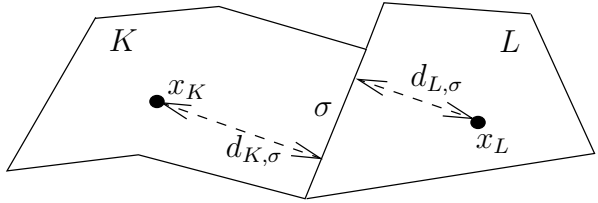
\caption{\label{fig:mesh}Notations associated with a polytopal mesh.}
\end{center}
\end{figure}

A discrete $H^1_0$ norm on $X_{\mathcal M}$ is defined by
\begin{equation}\label{discnorm}
||v||_{H^1_0,\mathcal M}^2:=\sum_{\sigma\in\mathcal E_{\mathcal M}} 
|\sigma| d_{\sigma}\left(\frac{v_K-v_L}{d_{\sigma}}\right)^2.
\end{equation}
Here, and in subsequent similar sums, we use the convention that
$K$ and $L$ are the cells on each side of $\sigma$, and that $v_L=0$
if $\sigma\subset \partial\O$ is a face of $K$. This
choice accounts for the homogeneous boundary conditions on $\partial \O$.

The major interest of the discrete $H^1_0$ norm, in view of the convergence steps
\ref{step1}--\ref{step3},
is appa\-rent in the two following theorems, proved by Eymard et al. \cite{EYM10}.
Theorem \ref{th-discP} is the key to reproduce at the discrete level the sequence
of inequalities \eqref{est1}--\eqref{energ} leading to the energy estimates
mentioned in Convergence Step \ref{step1} this requires suitable \emph{coercivity}
properties of the scheme. Theorem \ref{th-discR} covers Convergence Step \ref{step2}. Convergence step \ref{step3}
is more scheme-dependent, and relies on \emph{consistency} and \emph{limit-conformity} properties of the scheme.
Theorems \ref{th-discP} and \ref{th-discR} are examples of
discrete functional analysis results.

\begin{theorem}[Discrete Poincar\'e's inequality]\label{th-discP}
Let $\mathcal M$ be a mesh of $\O$ and set
\begin{equation}\label{defreg}
\theta_{\mathcal M}=\max\left\{\frac{d_{K,\sigma}}{d_{L,\sigma}}\,:\,
\sigma\in\mathcal E_{\mathcal M}\,,\;\mbox{$K,L$ cells on each side of $\sigma$}\right\}.
\end{equation}
If $\overline{\theta}\ge \theta_{\mathcal M}$, then there exists
$\ctel{cstP}$ only depending on $\overline{\theta}$ such that for any
$v\in X_{\mathcal M}$ we have $||v||_{L^2(\O)}\le \cter{cstP}||v||_{H^1_0,\mathcal M}$.
\end{theorem}

\begin{theorem}[Discrete Rellich's theorem]\label{th-discR}
Let $(\mathcal M_n)_{n\in\N}$ be a sequence of discretisations of $\O$ such
that $(\theta_{\mathcal M_n})_{n\in\N}$ is bounded and $h_{\mathcal M_n}\to 0$
as $n\to\infty$. If $v_n\in X_{\mathcal M_n}$
is such that $(||v_n||_{H^1_0,\mathcal M_n})_{n\in\N}$ is bounded, then
$(v_n)_{n\in\N}$ is relatively compact in $L^2(\O)$.
Furthermore, any limit in $L^2(\O)$ of a subsequence of $(v_n)_{n\in\N}$ belongs to
$H^1_0(\O)$.
\end{theorem}

\subsection{Examples}\label{ssec:ex}

Besides Theorems \ref{th-discP} and \ref{th-discR}, an important feature
of the discrete norm \eqref{discnorm} is its versatility;
it is suitable for numerous schemes, even with degrees of freedom that are not cell-centred.
Here we give a practical illustration, using two methods, of the usage
of Convergence Steps \ref{step1}--\ref{step3} and of the discrete norm \eqref{discnorm}.

\subsubsection{Two-point flux approximation finite volume scheme}\label{sec:tpfa}

The two-point flux approximation (\textsc{tpfa}) scheme for \eqref{base} is given by flux balances (obtained by
integrating \eqref{base} over the cells), and a finite difference approximation of
the flux $-\int_\sigma A(x)\nabla \bu(x)\cdot\n_K(x)d x$
using the two unknowns on each side of $\sigma$:
\begin{align}\label{balflux}
\forall K\in\mathcal M\,&:\,\sum_{\sigma\in\mathcal E_K} F_{K,\sigma}=\int_K f(x)d x,\\
\label{fluxtpfa}
\forall K\in\mathcal M\,,\;\forall\sigma\in\mathcal E_K\,&:\,
F_{K,\sigma}=\tau_\sigma (u_K-u_L).
\end{align}
Here, $\mathcal E_K$ is the set of faces of a cell $K\in\mathcal M$, and
the transmissivity $\tau_\sigma\in (0,\infty)$ depends on
$A$ and the local mesh geometry \cite{EGH00}. Under usual non-degeneracy assumptions on the
mesh, there exists $\ctel{ctaus}>0$ only depending on $\overline{a}$ and $\underline{a}$ such that
\begin{equation}\label{taus}
\tau_\sigma \ge \cter{ctaus}\frac{|\sigma|}{d_{\sigma}}.
\end{equation}

\paragraph{Convergence Step \ref{step1}}

The inequalities \eqref{est1}--\eqref{energ} that lead to
the a priori estimates on $\bu$ are obtained by the following sequence of
manipulations: (i) multiply \eqref{base}
by $v=\bu$ and integrate the resulting equation, (ii) apply Stokes' formula,
and (iii) use Poincar\'e's inequality.
Since the flux balance \eqref{balflux} is the discrete expression of \eqref{base}, we
reproduce these manipulations at the discrete level.

\begin{enumerate}
\item[(i)] \emph{Multiply and integrate}: we multiply \eqref{balflux} by $v_K=u_K$ and we sum on $K\in\mathcal M$. Accounting for \eqref{fluxtpfa} this gives
\begin{equation}\label{est0d}
\sum_{K\in\mathcal M}\sum_{\sigma\in \mathcal E_K} \tau_\sigma(u_K-u_L)u_K=
\sum_K \int_K f(x)d x\, u_K = \int_\O f(x)u(x)d x.
\end{equation}
\item[(ii)] \emph{Apply Stokes' formula}: this consists of gathering by faces the sum in the left-hand side
of \eqref{est1d}. The contributions of a face are
$\tau_\sigma(u_K-u_L)u_K$ and $\tau_\sigma(u_L-u_K)u_L=-\tau_\sigma(u_K-u_L)u_L$. Hence, using \eqref{taus}
and Cauchy-Schwarz' inequality on the right-hande side, we find
\begin{equation}\label{est1d}
\cter{ctaus}\sum_{\sigma\in\mathcal E_{\mathcal M}} \frac{|\sigma|}{d_{\sigma}}(u_K-u_L)^2\le
\sum_{\sigma\in\mathcal E_{\mathcal M}} \tau_\sigma(u_K-u_L)^2
\le ||f||_{L^2(\O)}||u||_{L^2(\O)}.
\end{equation}
\item[(iii)] \emph{Use Poincar\'e's inequality}: the left-hand side of \eqref{est1d} is
$\cter{ctaus}||u||_{H^1_0,\mathcal M}^2$. Invoking the discrete Poincar\'e's inequality (Theorem \ref{th-discP}),
we find $\ctel{csttpfa}$ only depending on an upper bound of $\theta_{\mathcal M}$ such that
\begin{equation}
||u||_{H^1_0,\mathcal M}\le \cter{csttpfa}||f||_{L^2(\O)}.
\label{energtpfa}\end{equation}
Estimate \eqref{energtpfa} is the discrete equivalent of \eqref{energ} for the solution
of the \textsc{tpfa} scheme.
\end{enumerate}

\paragraph{Convergence Step \ref{step2}}

This step is straightforward from \eqref{energtpfa} by using the discrete
Rellich's theorem. This estimate shows that if $(\mathcal M_n)_{n\in\N}$ is a sequence
of meshes as in Theorem \ref{th-discR} and if
$u_n$ is the solution of the \textsc{tpfa} scheme on $\mathcal M_n$, then $(||u_n||_{H^1_0,\mathcal M_n})_{n\in\N}$ remains bounded. Hence, up to a subsequence,
$u_n$ converges in $L^2(\O)$ towards some function $\bu\in H^1_0(\O)$.

\paragraph{Convergence Step \ref{step3}}

As mentioned above, proving that $\bu$ is the solution to \eqref{basew} hinges on
adequate consistency properties enjoyed by the scheme. Here, it all comes to 
the proper choice of transmissivities $\tau_\sigma$, and to the geometry of the
mesh. By taking $\varphi\in C^\infty_c(\O)$,
multiplying \eqref{balflux} for $\mathcal M=\mathcal M_n$ by $\varphi(x_K)$, and summing
over all $K$ we find
\[
\sum_{K\in\mathcal M_n}\sum_{\sigma\in\mathcal E_K}
\tau_\sigma  [(u_n)_K-(u_n)_L]\varphi(x_K) = \sum_{K\in\mathcal M_n}
\int_K f(x)\varphi(x_K)d x.
\]
We then gather the sums in the left-hand side by terms involving $(u_n)_K$:
\begin{equation}
\sum_{K\in\mathcal M_n}(u_n)_K\sum_{\sigma\in\mathcal E_K}
\tau_\sigma [\varphi(x_K)-\varphi(x_L)] = \sum_{K\in\mathcal M_n}
\int_K f(x)\varphi(x_K)d x
\label{cons1}\end{equation}
where $\varphi(x_L)=0$ if $\sigma\in\mathcal E_K$ lies on $\partial\O$.
The choice of $\tau_\sigma$, the geometrical assumptions constraining the
meshes for the \textsc{tpfa} method (that is, an orthogonality requirement of $(x_Kx_L)$
and $\sigma$ for a scalar product induced by $A^{-1}$), and the smoothness of $\varphi$ ensure that
$\sum_{\sigma\in\mathcal E_K}\tau_\sigma [\varphi(x_K)-\varphi(x_L)]
=-\int_K \div(A\nabla \varphi) + |K|\mathcal O(h_{\mathcal M_n})$,
where $|K|$ is the $d$-dimensional measure of $K$.
Relation \eqref{cons1} thus gives
\[
-\int_\O u_n(x) \div(A\nabla\varphi)(x)d x
+\mathcal O(||u_n||_{L^1(\O)}h_{\mathcal M_n})= \int_\O
f(x)\varphi(x)d x+\mathcal O(h_{\mathcal M_n}),
\]
where we used the smoothness of $\varphi$ in the right-hand side.
By the convergence of $u_n$ to $\bu$ in $L^2(\O)$, in the limit
$n\to\infty$ we find that $\bu\in H^1_0(\O)$ satisfies the following
property, classically equivalent to \eqref{basew}:
\[
\forall\varphi\in C^\infty_c(\O)\,,\;-\int_\O \bu(x) \div(A\nabla\varphi)(x)d x=\int_\O 
f(x)\varphi(x)d x.
\]

\begin{remark} The above reasoning apparently only shows the convergence of 
a \emph{subsequence} of $(u_n)_{n\in\N}$. However, since there is only one possible
limit (namely, the unique solution $\bu$ to \eqref{basew}),
this actually proves that the whole sequence $(u_n)_{n\in\N}$ converges to $\bu$.
\end{remark}

\subsubsection{Non-conforming $\mathbb{P}^1$ finite element}\label{sec:cr}

Usage of the discrete norm \eqref{discnorm} is not limited to numerical methods
with only/primarily cell unknowns. Let us consider a triangulation $\mathcal T$
of 2D poly\-gonal domain $\O$ (what follows also generalises to tetrahedral meshes
of a 3D polyhedral domain). The non-conforming Crouzeix-Raviart 
$\mathbb{P}^1$ finite ele\-ment \cite{CR} for \eqref{base} has degrees of freedom at the midpoints
$(\overline{x}_\sigma)_{\sigma\in\mathcal E_{\mathcal T}}$ of the triangulation's
edges. The discrete space $Y_{\mathcal T}$ of unknowns is
made of families of reals $u=(u_{\sigma})_{\sigma\in\mathcal E_{\mathcal T}}$, where
$u_\sigma=0$ if $\sigma\subset \partial\O$. These families are
identified with functions $u:\O\to \R$ that are piecewise linear on the mesh,
with values $(u_\sigma)_{\sigma\in\mathcal E_{\mathcal T}}$ at
$(\overline{x}_\sigma)_{\sigma\in\mathcal E_{\mathcal T}}$.
The non-conforming $\mathbb{P}^1$ approximation of \eqref{basew} is
\begin{equation}\label{basecr}
\left\{\begin{array}{l}
\mbox{Find $u\in Y_{\mathcal T}$ such that:}\\
\dsp \forall v\in Y_{\mathcal T}\,,\;\int_\O A(x)\nabla_b u(x)\cdot\nabla_b v(x)d x
=\int_\O f(x)v(x)d x
\end{array}\right.
\end{equation}
where $\nabla_b$ is the broken gradient: $(\nabla_b u)_{|K}$ is the constant gradient
of the linear function $u$ in the triangle $K\in\mathcal T$.

\paragraph{Convergence Step \ref{step1}}

To benefit from Theorems \ref{th-discP} and \ref{th-discR}, we need
to introduce the norm \eqref{discnorm}, which requires some choice of
cell unknowns. Here, the most natural choice is to set $u_K$ as the
value of $u$ at the centre of gravity $\overline{x}_K$ of $K$; since $u$ is linear
in $K$, this gives
\[
\forall K\in\mathcal T\,,\quad u_K=u(\overline{x}_K)=\frac{1}{3}\sum_{\sigma\in\mathcal E_K}u_\sigma.
\]
This choice associates (in a non-injective way) to each $u\in Y_{\mathcal T}$ a
$\widetilde{u}=(u_K)_{K\in\mathcal T}\in X_{\mathcal T}$. Two simple inequalities, both based on
the linearity of $u$ inside each triangle, will be useful to conclude
Convergence Step \ref{step1}.

\begin{lemma} Let $\eta_{\mathcal T}$ be the maximum over $K\in\mathcal T$
of the ratio of the exterior diameter of $K$ over the interior diameter of $K$. Assume
that $\overline{\eta}\ge \eta_{\mathcal T}$. Then there exists $\ctel{cstCR}$ only depending
on $\overline{\eta}$ such that, for all $u\in Y_{\mathcal T}$,
\begin{align}\label{estcr3}
||\widetilde{u}||_{H^1_0,\mathcal T} &\le \cter{cstCR}||\,|\nabla_b u|\,||_{L^2(\O)},\\
||\widetilde{u}-u||_{L^2(\O)}&\le h_{\mathcal T} ||\,|\nabla_b u|\,||_{L^2(\O)}.
\label{estcr4}\end{align}
\end{lemma}

\begin{proof}
Start with \eqref{estcr3}. There exists $\ctel{cstcr1}$ only
depending on $\overline{\eta}$ such that for all $\sigma\in\mathcal K$
we have ${\rm dist}(\overline{x}_K,\overline{x}_\sigma)\le \cter{cstcr1} d_\sigma$.
Hence, since $u$ is linear inside each triangle,
\begin{eqnarray}
\frac{|\widetilde{u}_K-\widetilde{u}_L|}{d_\sigma}&\le& \cter{cstcr1}
\frac{|u(\overline{x}_K)-u(\overline{x}_\sigma)|}{\mbox{dist}(\overline{x}_K,\overline{x}_\sigma)}
+\cter{cstcr1}\frac{|u(\overline{x}_L)-u(\overline{x}_\sigma)|}{\mbox{dist}(\overline{x}_L,\overline{x}_\sigma)}\nonumber\\
&\le& \cter{cstcr1}|(\nabla_b u)_{|K}|+\cter{cstcr1}|(\nabla_b u)_{|L}|
\label{estcr2}\end{eqnarray}
By squaring \eqref{estcr2}, multiplying by $|\sigma|d_\sigma$, summing over the
edges and using $\sum_{\sigma\in\mathcal E_K}|\sigma|d_\sigma
\le \ctel{cstcr2}|K|$ with $\cter{cstcr2}$ only depending on $\overline{\eta}$, we obtain \eqref{estcr3}.
The proof of \eqref{estcr4} is even simpler and follows directly from the fact that
$\widetilde{u}(x)-u(x)=u(\overline{x}_K)-u(x)=
(\nabla_b u)_{|K}\cdot(\overline{x}_K-x)$ for all $x\in K$. \end{proof}

Equipped with \eqref{estcr3} and \eqref{estcr4}, we now delve into
Convergence Step \ref{step1}. Substituting $v=u$ in the formulation \eqref{basecr} of the scheme,
the coercivity of $A$ entails
\[
\underline{a}||\,|\nabla_b u|\,||_{L^2(\O)}^2\le ||f||_{L^2(\O)}||u||_{L^2(\O)}.
\]
Using \eqref{estcr4} and $h_{\mathcal T}\le \mbox{diam}(\O)$,
this gives
\[
\underline{a}||\,|\nabla_b u|\,||_{L^2(\O)}^2\le ||f||_{L^2(\O)}(||\widetilde{u}||_{L^2(\O)}
+\mbox{diam}(\O)||\,|\nabla_b u|\,||_{L^2(\O)}).
\]
A bound on $\eta_{\mathcal T}$ implies a bound on $\theta_{\mathcal T}$ (defined by \eqref{defreg}).
Hence, the discrete Poincar\'e's inequality (Theorem \ref{th-discP}) and \eqref{estcr3}
lead to
\begin{equation}\label{energcr}
||\,|\nabla_b u|\,||_{L^2(\O)}\le (\cter{cstP}\cter{cstCR}+\mbox{diam}(\O))\underline{a}^{-1}
||f||_{L^2(\O)}.
\end{equation}
Estimate \eqref{energcr} is the discrete equivalent of the energy estimate \eqref{energ}. In conjunction
with \eqref{estcr3} it gives
\begin{equation}\label{energcr2}
||\widetilde{u}||_{H^1_0,\mathcal T}\le \cter{cstCR}(\cter{cstP}\cter{cstCR}+\mbox{diam}(\O))
\underline{a}^{-1}
||f||_{L^2(\O)}.
\end{equation}

\paragraph{Convergence Step \ref{step2}}

This is similar to the same step in the \textsc{tpfa} method. If $(\mathcal T_n)_{n\in\N}$
is a sequence of uniformly regular triangulations whose size tends to zero, then 
combining \eqref{energcr2} (with $\mathcal T=\mathcal T_n$) and the discrete Rellich's theorem (Theorem \ref{th-discR})
shows that $\widetilde{u}_n\to \bu$ in $L^2(\O)$ up to a subsequence,
for some $\bu\in H^1_0(\O)$. Moreover, by \eqref{estcr4} and \eqref{energcr},
we also have $u_n\to \bu$ in $L^2(\O)$.

\paragraph{Convergence Step \ref{step3}}

Assume now that 
\begin{equation}\label{asscr}
\nabla_b u_n\to \nabla \bu\mbox{ weakly in $L^2(\O)^d$ as $n\to\infty$.}
\end{equation}
For $\varphi\in C^\infty_c(\O)$ we define the interpolant $v_n\in Y_{\mathcal T_n}$
by $(v_n)_\sigma=\varphi(\overline{x}_\sigma)$. The smoothness of $\varphi$ ensures
that $v_n\to\varphi$ in $L^\infty(\O)$ and $\nabla_b v_n\to\nabla\varphi$ in $L^\infty(\O)^d$.
The convergence \eqref{asscr} therefore allows us to pass to the limit in
\eqref{basecr} written for $u_n$ and $v_n$. We deduce that $\bu\in H^1_0(\O)$ satisfies 
$\int_\O A\nabla\bu\cdot\nabla\varphi d x=\int_\O f\varphi d x$
for all smooth $\varphi$, which is equivalent to \eqref{basew}.

The proof of \eqref{asscr} relies on well-established techniques. By \eqref{energcr}
the sequence $(\nabla_b u_n)_{n\in\N}$ is bounded, and therefore converges weakly in $L^2(\O)^d$
to some $\bchi$,
up to a subsequence. We just need to prove that $\bchi= \nabla\bu$. Take $\bpsi\in C^\infty_c(\O)^d$ and, by Stokes' formula in each triangle,
\begin{multline}
\int_\O \nabla_b u_n(x)\cdot\bpsi(x)d x=\sum_{K\in\mathcal T_n}\int_K \nabla_b u_n(x)\cdot\bpsi(x)d x
\\
=\sum_{K\in\mathcal T_n}\int_{\partial K} (u_n)_{|K}(x) \n_K\cdot\bpsi(x)d S(x)
-\sum_{K\in\mathcal T_n}\int_K u_n(x) \div\bpsi(x)d x\\
=Z_n -\int_\O u_n(x) \div\bpsi(x)d x,
\label{estcr33}\end{multline}
where $\n_K$ is the outer normal to $K$ and $(u_n)_{|K}$ denotes values on $\sigma$
from $K$. Since $\bpsi=0$ on $\partial\O$ and
$\bpsi\cdot\n_K+ \bpsi\cdot\n_L=0$ on the interface $\sigma$ between $K$ and $L$, we have
\begin{multline*}
\sum_{K\in\mathcal T_n}\sum_{\sigma\in\mathcal E_K}\int_{\sigma} (u_n)_\sigma \n_K\cdot\bpsi(x)d S(x)
\\
=\sum_{\sigma\in\mathcal E,\,\sigma\subset\O} \int_\sigma (u_n)_\sigma (\n_K\cdot\bpsi(x)
+\n_L\cdot\bpsi(x))d S(x)=0.
\end{multline*}
and thus
\[
Z_n=\sum_{K\in\mathcal T_n}\sum_{\sigma\in\mathcal E_K}\int_{\sigma} [(u_n)_{|K}(x)-(u_n)_\sigma] \n_K\cdot\bpsi(x)d S(x).
\]
By definition of $(u_n)_{\sigma}$ we have $\int_\sigma [(u_n)_{|K}(x)-(u_n)_\sigma]d S(x)=0$.
Using $|(u_n)_{|K}-(u_n)_\sigma|\le {\rm diam}(K) |(\nabla_b u_n)_{|K}|$ and
the smoothness of $\bpsi$, we infer
\begin{multline*}
|Z_n|=\left|\sum_{K\in\mathcal T_n}\sum_{\sigma\in\mathcal E_K}\int_{\sigma} [(u_n)_{|K}(x)-(u_n)_\sigma] \n_K\cdot[\bpsi(x)
-\bpsi(\overline{x}_\sigma)]d S(x)\right|\\
\le C_\bpsi h_{\mathcal M_n}\sum_{K\in\mathcal T_n}\sum_{\sigma\in\mathcal E_K}|\sigma| h_K |(\nabla_b u_n)_{|K}|
\le 3C_\bpsi \cter{cstcr44} h_{\mathcal M_n} ||\,|\nabla_b u_n|\,||_{L^1(\O)}
\end{multline*}
with $\ctel{cstcr44}$ not depending on $n$ (we used the regularity assumption on $\mathcal T_n$ to write $|\sigma|h_K\le \cter{cstcr44}|K|$).
Invoking the discrete energy estimate \eqref{energcr}, we deduce that $Z_n\to 0$ and we
therefore evaluate the limit of \eqref{estcr33} since $u_n\to u$ in $L^2(\O)$
and $\nabla_b u_n\to \bchi$ weakly in $L^2(\O)^d$. This gives 
$\int_\O \bchi(x)\cdot\bpsi(x)d x=-\int_\O \bu(x)\div \bpsi(x)d x$,
which proves that $\bchi=\nabla\bu$ as required.

\section{Extension to non-linear models}\label{sec:nl}

The previous technique, based on the convergence steps \ref{step1}--\ref{step3} and
on the discrete Rellich's theorem and the discrete Poincare's inequality, would not be very useful
if it only applied to the linear diffusion equation \eqref{base}. Convergence of
numerical methods for this equation is well-known, and best obtained through
error estimates. The power of the compactness techniques presented above is
that they seamlessly apply to non-linear models, including models of physical
relevance such as oil recovery and the Navier--Stokes equations.
Presenting a complete review of these techniques on such models is beyond the
scope of this article, but we can give an overview of some of the latest developments
in this area.

\subsection{Stationary equations}

\subsubsection{Academic example}

We first show with an academic example how to apply the previous techniques
to a non-linear model. We consider
\begin{equation}\label{basenl}
\left\{
\begin{array}{ll}
-\div(A(\cdot,\bu)\nabla \bu)=F(\bu)&\mbox{ in $\O$},\\
\bu=0&\mbox{ on $\partial\O$}
\end{array}\right.
\end{equation}
where $F:\R\mapsto\R$ is continuous and bounded, and $A:\O\times \R\mapsto \mathcal M_d(\R)$
is a Caratheodory function (measurable with respect to $x\in\O$, continuous with respect to $s\in\R$)
such that for all $s\in\R$ the function $A(\cdot,s)$ satisfies \eqref{hypa} with $\underline{a}$
and $\overline{a}$ not depending on $s$. The weak form of \eqref{basenl} consists
of \eqref{basew} with $f(x)$ and $A(x)$ replaced with
$F(\bu(x))$ and $A(x,\bu(x))$, respectively.

As in the linear model case, establishing the convergence of a numerical method
for \eqref{basenl} by using discrete functional analysis techniques consists of mimicking estimates
on the continuous equation. Here, these estimates are obtained as for the
linear model; substituting $v=\bu$ in the weak form of \eqref{basenl} and using
the coercivity of $A$, the bound on $F$ and Poincar\'e's inequality,
it is seen that $\bu$ satisfies
\[
||\bu||_{H^1_0(\O)}\le {\mbox{diam}(\O)}{\underline{a}}^{-1}|\O|^{1/2}||F||_{L^\infty(\R)}.
\]

Writing a numerical method for \eqref{basenl} using a method for the linear
equation \eqref{base} is usually quite straightforward: all $f(x)$ and $A(x)$ appearing in the definition
of the method (e.g. through $\tau_\sigma$ for the \textsc{tpfa} method) have to be replaced with $F(u(x))$
and $A(x,u(x))$, where $u$ is the approximation sought through the scheme. A quick inspection
of Convergence Steps \ref{step1} in Sections \ref{sec:tpfa} and \ref{sec:cr} shows that
the discrete energy estimates \eqref{energtpfa}, \eqref{energcr} and \eqref{energcr2}
hold with $||f||_{L^2(\O)}$ replaced with $|\Omega|^{1/2}||F||_{L^\infty(\R)}$.

Convergence Step \ref{step2} then follows from Theorem \ref{th-discR} exactly as in the linear case,
and we find $\bu\in H^1_0(\O)$ such that up to a subsequence $u_n\to \bu$ in $L^2(\O)$.
This ensures that $F(u_n)\to F(\bu)$ in $L^2(\O)$,
and that up to a subsequence $A(\cdot,u_n)\to A(\cdot,u)$ almost everywhere while remaining uniformly bounded.
These convergences enable us to evaluate the limit of the scheme by
following the exact same technique as in Convergence Steps \ref{step3} for the linear model.
This establishes that $\bu$ is a weak solution of \eqref{basenl}.

\begin{remark} Although the strong convergence of $u_n$ to $\bu$ is not necessary in the linear case (weak convergence would suffice),
it is essential for non-linear models such as \eqref{basenl}. Indeed, if $(u_n)_{n\in\N}$ only
converges weakly, then $F(u_n)$ and $A(\cdot,u_n)$ may not converge to the correct
limits $F(\bar u)$ and $A(\cdot,\bar u)$.
\end{remark}

\subsubsection{Physical models}

As mentioned in the introduction, the strength of a convergence analysis via
compactness techniques is that it applies to fully non-linear models that are relevant
in a number of applications.

\paragraph{Elliptic equations with measure data}

Equations of the form \eqref{base} appear in models of
oil recovery, in which $f$ models
wells. The relative scales of the reservoir and the wellbores justifies
taking a Radon measure for this source term \cite{fg00,DT}. The ensuing analysis is
more complex. To start with, the weak formulation \eqref{basew} is no longer suitable
\cite{BBGGPV95,DMOP99}. Moreover, due to the singularity of the source term, the
solution has very weak regularity properties, and may not be unique. This prevents
any proof of error estimates for numerical approximations of these models.

Discrete functional analysis tools were developed to establish the convergence of the
\textsc{tpfa} finite volume scheme for diffusion and (possibly non-coercive) convection--diffusion
equations with measures as source terms \cite{GAL99,DRO03}.  
Key elements to obtaining a priori estimates on the solutions to these
equations are the Sobolev spaces $W^{1,p}_0(\O)$ (which is $H^1_0(\O)$ if $p=2$),
and the Sobolev embeddings. The corresponding numerical analysis requires
the discrete $W^{1,p}_0$ norm on $X_{\mathcal M}$
\[
||v||_{W^{1,p}_0,\mathcal M}^p:=\sum_{\sigma\in\mathcal E_{\mathcal M}} 
|\sigma| d_{\sigma}\left(\frac{v_K-v_L}{d_{\sigma}}\right)^p,
\]
to generalise the discrete Poincar\'e's and Rellich's theorems to this norm,
and to establish discrete Sobolev embeddings: if $p\in (1,d)$ and $q\le \frac{dp}{d-p}$
then
\begin{equation}\label{sobo}
||v||_{L^q(\O)}\le C||v||_{W^{1,p}_0,\mathcal M}.
\end{equation}

\begin{remark} The most efficient proofs of the discrete Poincar\'e's and Rellich's theorems 
actually use the discrete Sobolev embeddings \cite{EYM10,koala}.
\end{remark}

\begin{remark} The numerical study of \eqref{base} with $f$ measure is currently
(mostly) limited to the \textsc{tpfa} scheme, since no other method 
has in general the
structure that enables the mimicking of the continuous estimates \cite{review}.
\end{remark}




\paragraph{Leray--Lions and $p$-Laplace equations}

These models are non-linear gene\-ralisations of \eqref{base}, that appear in models of gaciology
\cite{GR}. They have a more severe non-linearity than \eqref{basenl}, since they
involve both $\bu$ and $\nabla\bu$. The general form of these equations is obtained by
replacing $\div(A\nabla \bu)$ in \eqref{base} with $\div(a(\cdot,\bu,\nabla\bu))$,
where $a:\O\times \R\times\R^d\mapsto\R^d$ satisfies growth, monotony and coer\-civity assumptions.
The simplest form is probably the $p$-Laplace equation $-\div(|\nabla\bu|^{p-2}\nabla \bu)=f$ for $p\in (1,\infty)$.

Uniqueness may fail for these equations \cite[Remark 3.4]{DRO12}, which completely prevents classical error estimates for their numerical approximations.
Compactness techniques were used to study the convergence of at least three different
schemes for Leary--Lions equations: the mixed finite volume method \cite{DRO06-II},
the discrete duality finite volume method \cite{AND07}, and a cell-centred finite volume scheme
\cite{EYM09}. These studies make use of discrete scheme-dependent $W^{1,p}_0$ norms and
related discrete Rellich's and Poincar\'e's theorems. They also require an (easy) adaptation
to the discrete setting of Minty's monotony method, to deal with the non-linearity
involving $\nabla\bu$.

\subsection{Time-dependent and Navier--Stokes equations}

Studying non-linear time-dependent models requires space--time compactness results.
In the context of Sobolev spaces, these results are usually 
variants of the Aubin--Simon theorem \cite{aubin,simon} which, roughly speaking,
ensures the compactness in $L^p(\O\times (0,T))$ of a sequence $(u_n)_{n\in\N}$ provided
that $(\nabla u_n)_{n\in\N}$ is bounded in $L^p(\O\times (0,T))^d$ and that $(\partial_t u_n)_{n\in\N}$ is bounded
in $L^q(0,T;W^{-1,r}(\O))$, where $W^{-1,r}(\O)=(W^{1,r'}_0(\O))'$.
These are natural spaces in which solutions to parabolic \textsc{pde}s can be estimated.

Carrying out the numerical analysis of these equations with irregular
data necessitates the development of discrete versions of the Aubin--Simon theorem; this often
includes designing a discrete dual norm mimicking the norm in $W^{-1,r'}(\O)$. This analysis
has been done for various schemes and models: transient Leray--Lions equations \cite{DRO12}, including non-local dependencies of $a(x,\bu,\nabla\bu)$ with respect to $\bu$
(as in image segmentation \cite{EYM11-2}); a model of miscible fluid flows in porous media from
oil recovery \cite{CHA07,CKM15}; Stefan's model of melting material \cite{EYM13-2};
Richards' model and multi-phase flows in porous media \cite{zamm2013}.
Discrete Aubin--Simon theorems also sometimes need to be completed with other
compactness results, such as compactness results involving
sequences of discrete spaces \cite{gal-12-com}, or discrete compensated compactness theorems \cite{DE15} to deal with degenerate parabolic \textsc{pde}s.

All these compactness results only provide strong convergence in
a space--time averaged norm (e.g. $L^p(\O\times (0,T))$ for some $p<\infty$).
However, Droniou et al. \cite{DE15,DE-fvca7} recently developed a technique
to establish a uniform-in-time convergence result (i.e. in $L^\infty(0,T;L^2(\O))$) by combining the initial averaged convergences,
energy estimates from the \textsc{pde}, and a discontinuous weak Ascoli--Arzela theorem.
This strong uniform convergence corresponds to the needs of end-users, who are usually more interested in the behaviour of the solution at the final time rather than averaged over time.

\paragraph{Navier--Stokes equations}

The regularity and uniqueness of the solution to Navier--Stokes equations is a famous open problem.
Therefore, as explained in the introduction, the convergence analysis of numerical schemes for
these equations cannot be based on error estimates.
If it is to be rigorously carried out under reasonable physical assumptions, 
this convergence analysis can only be done
through compactness techniques.

Let us first consider the continuous case.
Because of the term $(\bu\cdot\nabla)\bu$ in 
\begin{equation}\label{eq:ns}
\partial_t \bu -\Delta \bu + (\bu\cdot\nabla) \bu +\nabla\overline{p}=f,
\end{equation}
evaluating the limit from a sequence of approximate solutions $(u_n)_{n\in\N}$ requires a strong
space--time $L^2$ compactness on $(u_n)_{n\in\N}$ (since
$(\nabla u_n)_{n\in\N}$ converges only in $L^2(\O\times (0,T))^d$-weak). Kolmogorov's
theorem ensures this strong compactness on $(u_n)_{n\in\N}$ provided that we can
control the space-translates and time-translates of the functions. The space translates
are naturally estimated thanks to the bound on $(\nabla u_n)_{n\in\N}$, and the
the time translates $||u_n(\cdot+\tau,\cdot)-u_n||_{L^1(0,T;L^2(\O))}$ are
estimated by
\[
\int_\O |u_n(t+\tau,x)-u_n(t,x)|^2d x
=\int_\O \int_t^{t+\tau} \partial_t u_n(s,x) (u_n(t+\tau,x)-u_n(t,x))d xd s.
\]
Equation \eqref{eq:ns} is then used to substitute $\partial_t u_n$ in terms of $u_n$
and its space derivatives (since $\div u_n=0$, the term involving $\nabla {p}_n$ disappears).
Bounding the term $(u_n\cdot\nabla )u_n\times u_n$ that appears after this substitution
requires Sobolev estimates on $(u_n)_{n\in\N}$; these ensure that, only considering the space
integral, $u_n\in L^6(\O)$ and thus $|u_n|^2 |\nabla u_n|\in L^{6/5}(\O)$
(wihout Sobolev estimates, $u_n\in L^2(\O)$ and $|u_n|^2|\nabla u_n|$ is not even integrable).

The same issue arises in the convergence analysis of numerical methods for Navier--Stokes
equations. Discrete Sobolev estimates of the kind \eqref{sobo} are required to estimate
the time-translates of the approximate solutions and ensure the convergence towards
the correct model. Droniou and Eymard \cite{DRO09} did this for the mixed finite volume method, and Chenier et al. \cite{CEGH} considered an extension of the marker-and-cell (\textsc{mac}) scheme; both references
establish more scheme-specific Sobolev embeddings than \eqref{sobo}, but this
general inequality is actually sufficient for the analyses carried out in these works.

\section{Conclusions and perspectives}\label{sec:ccl}

We presented techniques that enable the convergence analysis of numerical
schemes for \textsc{pde}s under assumptions that are compatible with field applications.
In particular, discontinuous coefficients or fully non-linear physically relevant
models can be handled. These techniques do not require the uniqueness or regularity of the
solutions, and are based on discrete functional analysis tools -- that is the translation
to the discrete setting of the functional analysis used in the study of the
\textsc{pde}s.

These discrete tools were adapted to a number of schemes, including
the hybrid mixed mimetic family \cite{DRO10} (which contains the hybrid
finite volumes \cite{EYM10}, the mimetic finite differences \cite{BRE05-II}, and the mixed finite
volumes \cite{DRO06}), the discrete duality finite volumes \cite{AND07},
the discontinuous Galerkin methods \cite{DPE12}.

It might appear from our brief introduction that the discrete
Sobolev norms and all related results (Poincar\'e, Rellich, etc.)
require specific adaptations for each scheme or model. This is usually not the case.
A framework was recently designed, the gradient scheme framework \cite{EYM12,DRO12,koala},
that enables the unified convergence analysis of many different schemes for many
diffusion \textsc{pde}s. The idea is to identify a set of five properties that are not related to any
model, but are intrinsic to the discrete space and operators (gradient, etc.) of the numerical
methods; convergence proofs of numerical approximations of many different models can be
carried out based on these five properties only (sometimes even fewer). Generic discrete functional analysis tools exist to ensure that
several well-known schemes -- including meshless methods -- satisfy these properties \cite{DEH15}, and therefore that the aforementioned
convergence results apply to these schemes. 
The gradient scheme framework covers several boundary conditions, and also guided the design
of new schemes \cite{EYM12,DEF14}.

\bibliographystyle{plain}
\bibliography{ctac14-proceedings-droniou}

\begin{thebibliography}{10}

\bibitem{AND07}
B.~Andreianov, F.~Boyer, and F.~Hubert.
\newblock Discrete duality finite volume schemes for {L}eray-{L}ions-type
  elliptic problems on general 2{D} meshes.
\newblock {\em Numer. Methods Partial Differential Equations}, 23(1):145--195,
  2007.
\newblock DOI: 10.1002/num.20170.

\bibitem{aubin}
J-.P. Aubin.
\newblock Un th\'eor\`eme de compacit\'e.
\newblock {\em C. R. Math. Acad. Sci. Paris}, 256:5042--5044, 1963.

\bibitem{BBGGPV95}
P.~B{\'e}nilan, L.~Boccardo, T.~Gallou{\"e}t, R.~Gariepy, M.~Pierre, and J.~L.
  V{\'a}zquez.
\newblock An {$L^1$}-theory of existence and uniqueness of solutions of
  nonlinear elliptic equations.
\newblock {\em Ann. Scuola Norm. Sup. Pisa Cl. Sci. (4)}, 22(2):241--273, 1995.
\newblock URL: http://www.numdam.org/item?id=ASNSP\_1995\_4\_22\_2\_241\_0.

\bibitem{brezis}
H.~Brezis.
\newblock {\em Functional analysis, {S}obolev spaces and partial differential
  equations}.
\newblock Universitext. Springer, New York, 2011.

\bibitem{BRE05-II}
F.~Brezzi, K.~Lipnikov, and V.~Simoncini.
\newblock A family of mimetic finite difference methods on polygonal and
  polyhedral meshes.
\newblock {\em Math. Models Methods Appl. Sci.}, 15(10):1533--1551, 2005.
\newblock DOI: 10.1142/S0218202505000832.

\bibitem{CHA07}
C.~Chainais-Hillairet and J.~Droniou.
\newblock Convergence analysis of a mixed finite volume scheme for an
  elliptic-parabolic system modeling miscible fluid flows in porous media.
\newblock {\em SIAM J. Numer. Anal.}, 45(5):2228--2258, 2007.
\newblock DOI: 10.1137/060657236.

\bibitem{CKM15}
C.~Chainais-Hillairet, S.~Krell, and A.~Mouton.
\newblock Convergence analysis of a {DDFV} scheme for a system describing
  miscible fluid flows in porous media.
\newblock {\em Numer. Methods Partial Differential Equations}, 31(3):723--760,
  2015.
\newblock DOI: 10.1002/num.21913.

\bibitem{CEGH}
E.~Ch\'enier, R.~Eymard, T.~Gallou\"et, and R.~Herbin.
\newblock An extension of the {MAC} scheme to locally refined meshes:
  convergence analysis for the full tensor time-dependent {N}avier--{S}tokes
  equations.
\newblock {\em Calcolo}, 52(1):69--107, 2015.
\newblock DOI: 10.1007/s10092-014-0108-x.

\bibitem{CR}
M.~Crouzeix and P.-A. Raviart.
\newblock Conforming and nonconforming finite element methods for solving the
  stationary {S}tokes equations. {I}.
\newblock {\em Rev. Fran\c caise Automat. Informat. Recherche Op\'erationnelle
  S\'er. Rouge}, 7(R-3):33--75, 1973.

\bibitem{DMOP99}
G.~Dal~Maso, F.~Murat, L.~Orsina, and A.~Prignet.
\newblock Renormalized solutions of elliptic equations with general measure
  data.
\newblock {\em Ann. Scuola Norm. Sup. Pisa Cl. Sci. (4)}, 28(4):741--808, 1999.
\newblock URL: http://www.numdam.org/item?id=ASNSP\_1999\_4\_28\_4\_741\_0.

\bibitem{DIP13-2}
D.~Di~Pietro and M.~Vohralik.
\newblock A review of recent advances in discretization methods, a posteriori
  error analysis, and adaptive algorithms for numerical modeling in
  geosciences.
\newblock {\em Oil \& Gas Science and Technology}, 69(4):701--730, 2014.
\newblock DOI: 10.2516/ogst/2013158.

\bibitem{DPE12}
D.~A. Di~Pietro and A.~Ern.
\newblock {\em Mathematical aspects of discontinuous {G}alerkin methods},
  volume~69 of {\em Mathematics \& Applications (Berlin)}.
\newblock Springer, Heidelberg, 2012.
\newblock DOI: 10.1007/978-3-642-22980-0.

\bibitem{DRO06-II}
J.~Droniou.
\newblock Finite volume schemes for fully non-linear elliptic equations in
  divergence form.
\newblock {\em M2AN Math. Model. Numer. Anal.}, 40(6):1069--1100 (2007), 2006.
\newblock DOI: 10.1051/m2an:2007001.

\bibitem{review}
J.~Droniou.
\newblock Finite volume schemes for diffusion equations: introduction to and
  review of modern methods.
\newblock {\em Math. Models Methods Appl. Sci. (M3AS)}, 24(8):1575--1619, 2014.
\newblock DOI: 10.1142/S0218202514400041.

\bibitem{DRO06}
J.~Droniou and R.~Eymard.
\newblock A mixed finite volume scheme for anisotropic diffusion problems on
  any grid.
\newblock {\em Numer. Math.}, 105(1):35--71, 2006.
\newblock DOI: 10.1007/s00211-006-0034-1.

\bibitem{DRO09}
J.~Droniou and R.~Eymard.
\newblock Study of the mixed finite volume method for {S}tokes and
  {N}avier--{S}tokes equations.
\newblock {\em Numer. Methods Partial Differential Equations}, 25(1):137--171,
  2009.
\newblock DOI: 10.1002/num.20333.

\bibitem{DE15}
J.~Droniou and R.~Eymard.
\newblock Uniform-in-time convergence of numerical methods for non-linear
  degenerate parabolic equations.
\newblock {\em Numer. Math.}, 2015.
\newblock DOI: 10.1007/s00211-015-0733-6.

\bibitem{DEF14}
J.~Droniou, R.~Eymard, and P.~F\'eron.
\newblock Gradient schemes for {S}tokes problem.
\newblock {\em IMA J. Numer. Anal.}, 2015.
\newblock To appear.

\bibitem{koala}
J.~Droniou, R.~Eymard, T.~Gallou\"et, C.~Guichard, and R.~Herbin.
\newblock Gradient schemes for elliptic and parabolic problems.
\newblock 2015.
\newblock In preparation.

\bibitem{DRO10}
J.~Droniou, R.~Eymard, T.~Gallou{\"e}t, and R.~Herbin.
\newblock A unified approach to mimetic finite difference, hybrid finite volume
  and mixed finite volume methods.
\newblock {\em Math. Models Methods Appl. Sci.}, 20(2):265--295, 2010.
\newblock DOI: 10.1142/S0218202510004222.

\bibitem{DRO12}
J.~Droniou, R.~Eymard, T.~Gallou{\"e}t, and R.~Herbin.
\newblock Gradient schemes: a generic framework for the discretisation of
  linear, nonlinear and nonlocal elliptic and parabolic equations.
\newblock {\em Math. Models Methods Appl. Sci. (M3AS)}, 23(13):2395--2432,
  2012.
\newblock DOI: 10.1142/S0218202513500358.

\bibitem{DE-fvca7}
J.~Droniou, R.~Eymard, and C.~Guichard.
\newblock Uniform-in-time convergence of numerical schemes for {R}ichards' and
  {S}tefan's models.
\newblock In M.~Ohlberger J.~Fuhrmann and C.~Rohde Eds., editors, {\em Finite
  Volumes for Complex Applications VII -- Methods and Theoretical Aspects},
  volume~77, pages 247--254. Springer, 2014.
\newblock DOI: 10.1007/978-3-319-05684-5\_23.

\bibitem{DEH15}
J.~Droniou, R.~Eymard, and R.~Herbin.
\newblock Gradient schemes: generic tools for the numerical analysis of
  diffusion equations.
\newblock {\em M2AN Math. Model. Numer. Anal.}, 2015.
\newblock To appear.

\bibitem{DRO03}
J.~Droniou, T.~Gallou{\"e}t, and R.~Herbin.
\newblock A finite volume scheme for a noncoercive elliptic equation with
  measure data.
\newblock {\em SIAM J. Numer. Anal.}, 41(6):1997--2031, 2003.
\newblock DOI: 10.1137/S0036142902405205.

\bibitem{DPP03}
J.~Droniou, A.~Porretta, and A.~Prignet.
\newblock Parabolic capacity and soft measures for nonlinear equations.
\newblock {\em Potential Anal.}, 19(2):99--161, 2003.
\newblock DOI: 10.1023/A:1023248531928.

\bibitem{DT}
J.~Droniou and K.~S. Talbot.
\newblock On a miscible displacement model in porous media flow with measure
  data.
\newblock {\em SIAM J. Math. Anal.}, 46(5):3158--3175, 2014.
\newblock DOI: 10.1137/130949294.

\bibitem{EYM13-2}
R.~Eymard, P.~F\'eron, T.~Gallou\"et, R.~Herbin, and C.~Guichard.
\newblock Gradient schemes for the {S}tefan problem.
\newblock {\em IJFV International Journal On Finite Volumes}, 10, 2013.
\newblock URL: http://www.i2m.univ-amu.fr/IJFV/spip.php?article47.

\bibitem{EGH00}
R.~Eymard, T.~Gallou{\"e}t, and R.~Herbin.
\newblock Finite volume methods.
\newblock In P.~G. Ciarlet and J.-L. Lions, editors, {\em Techniques of
  Scientific Computing, Part III}, Handbook of Numerical Analysis, VII, pages
  713--1020. North-Holland, Amsterdam, 2000.

\bibitem{EYM09}
R.~Eymard, T.~Gallou{\"e}t, and R.~Herbin.
\newblock Cell centred discretisation of non linear elliptic problems on
  general multidimensional polyhedral grids.
\newblock {\em J. Numer. Math.}, 17(3):173--193, 2009.
\newblock DOI: 10.1515/JNUM.2009.010.

\bibitem{EYM10}
R.~Eymard, T.~Gallou{\"e}t, and R.~Herbin.
\newblock Discretization of heterogeneous and anisotropic diffusion problems on
  general nonconforming meshes {SUSHI}: a scheme using stabilization and hybrid
  interfaces.
\newblock {\em IMA J. Numer. Anal.}, 30(4):1009--1043, 2010.
\newblock DOI: 10.1093/imanum/drn084.

\bibitem{EYM12}
R.~Eymard, C.~Guichard, and R.~Herbin.
\newblock Small-stencil 3{D} schemes for diffusive flows in porous media.
\newblock {\em ESAIM Math. Model. Numer. Anal.}, 46(2):265--290, 2012.
\newblock DOI: 10.1051/m2an/2011040.

\bibitem{zamm2013}
R.~Eymard, C.~Guichard, R.~Herbin, and R.~Masson.
\newblock Gradient schemes for two-phase flow in heterogeneous porous media and
  {R}ichards equation.
\newblock {\em ZAMM Z. Angew. Math. Mech.}, 94(7-8):560--585, 2014.
\newblock DOI: 10.1002/zamm.201200206.

\bibitem{EYM11-2}
R.~Eymard, A.~Handlovi{\v{c}}ov{\'a}, R.~Herbin, K.~Mikula, and
  O.~Sta{\v{s}}ov{\'a}.
\newblock Gradient schemes for image processing.
\newblock In {\em Finite volumes for complex applications {VI} - Problems \&
  Perspectives}, volume~4 of {\em Springer Proc. Math.}, pages 429--437.
  Springer, Heidelberg, 2011.
\newblock DOI: 10.1007/978-3-642-20671-9\_45.

\bibitem{fg00}
P.~Fabrie and T.~Gallou\"et.
\newblock Modelling wells in porous media flow.
\newblock {\em Math. Models Methods Appl. Sci.}, 10(5):673--709, 2000.
\newblock DOI: 10.1142/S0218202500000367.

\bibitem{phdfaille}
I.~Faille.
\newblock {\em Mod\'elisation bidimensionnelle de la gen\`ese et de la
  migration des hydrocarbures dans un bassin s\'edimentaire}.
\newblock PhD thesis, Universit\'e Joseph Fourier -- Grenoble 1, 1992.

\bibitem{F95}
X.~Feng.
\newblock On existence and uniqueness results for a coupled system modeling
  miscible displacement in porous media.
\newblock {\em J. Math. Anal. Appl.}, 194(3):883--910, 1995.
\newblock DOI: 10.1006/jmaa.1995.1334.

\bibitem{GAL99}
T.~Gallou{\"e}t and R.~Herbin.
\newblock Finite volume approximation of elliptic problems with irregular data.
\newblock In {\em Finite volumes for complex applications {II}}, pages
  155--162. Hermes Sci. Publ., Paris, 1999.

\bibitem{gal-12-com}
T.~Gallou{\"e}t and J.~C. Latch{\'e}.
\newblock Compactness of discrete approximate solutions to parabolic {PDE}s --
  application to a turbulence model.
\newblock {\em Commun. Pure Appl. Anal}, 12(6):2371--2391, 2012.
\newblock DOI: 10.3934/cpaa.2012.11.2371.

\bibitem{GR}
R.~Glowinski and J.~Rappaz.
\newblock Approximation of a nonlinear elliptic problem arising in a
  non-newtonian fluid flow model in glaciology.
\newblock {\em M2AN Math. Model. Numer. Anal.}, 37(1):175--186, 2003.
\newblock DOI: 10.1051/m2an:2003012.

\bibitem{simon}
J.~Simon.
\newblock Compact sets in {$L^p(0,T;B)$}.
\newblock {\em Annali Mat. Pura appl. (IV)}, CXLVI:65--96, 1987.
\newblock DOI: 10.1007/BF01762360.

\end{thebibliography}

\end{document}